\documentclass{amsart}

\usepackage{booktabs}
\usepackage{float}

\newtheorem{theorem}{Theorem}[section]
\newtheorem{lemma}[theorem]{Lemma}

\theoremstyle{definition}
\newtheorem{definition}[theorem]{Definition}
\newtheorem{proposition}[theorem]{Proposition}
\newtheorem{corollary}[theorem]{Corollary}
\newtheorem{example}[theorem]{Example}

\theoremstyle{remark}
\newtheorem{remark}[theorem]{Remark}

\newcommand{\lcm}{lcm}
\newcommand{\Aut}{Aut}

\numberwithin{equation}{section}

\begin{document}
    
    \title{Twisted Frobenius-Schur Indicators and Character Degree Sums in Dihedral Groups}
    
    \author[Yerrapati]{Venkata Subbaiah Yerrapati}
    \address{(Yerrapati) Department of Mathematics, S. V. National Institute of Technology, Surat-7, Gujarat, India}
    \email{yvsmath@gmail.com}
    \thanks{The first author would like to thank the Department of Education, Government of India, for the financial assistance.}
    
    \author[Dixit]{Rahul Dixit}
    \address{(Dixit) Department of Artificial Intelligence, S. V. National Institute of Technology, Surat-7, Gujarat, India}
    \email{rahuldixit@aid.svnit.ac.in}
    
    \author[Shukla]{Ajay Kumar Shukla}
    \address{(Shukla) Department of Mathematics, S. V. National Institute of Technology, Surat-7, Gujarat, India}
    \email{aks@amhd.svnit.ac.in}
    
    
    \date{\today}
    
    
    
    \begin{abstract}
        Let $G$ be a finite group and $T(G)$ be the sum of the degrees of its irreducible complex representations. We investigate the relationship between $T(G)$ and the number of twisted involutions $m_\sigma = |\{g \in G \mid \sigma(g) = g^{-1}\}|$ for an automorphism $\sigma$. While it is known that $T(G) = m_e$ for the identity automorphism $e$ in certain cases (e.g., real characters), we analyze this relation for non-identity automorphisms of groups of order $p, 2p, p^2$. We prove that for the family of Dihedral groups $D_n$, the inequality $T(D_n) \geq m_\sigma$ holds for all $\sigma \in \Aut(D_n)$. We provide a complete classification of $m_\sigma$ using number-theoretic properties of the automorphism parameters.
    \end{abstract}

    \maketitle
    
    \section{Introduction and Preliminaries}
	
	Let $G$ be a finite group and let $\Aut(G)$ be the automorphism group of $G$. For a fixed automorphism $\sigma \in \Aut(G)$, we consider the following set 
	\begin{align*}
		S_{\sigma} = \{\pi \in G \ | \ \sigma(\pi) = \pi^{-1}\},
	\end{align*}
	such sets arise naturally when studying involutions twisted by automorphisms. Frobenius and Schur \cite{zbMATH02646565} introduced Frobenius-Schur indicator in the studying of real representations of finite group. For an irreducible representation $(\pi, V)$ of finite group $G$ with character $\chi$, the corresponding Frobenius-Schur indicator is defined by
    \begin{align}
        \epsilon(\pi) = \frac{1}{|G|}\sum_{g \in G}\chi(g^2)
    \end{align}
    Indeed, they showed that $\epsilon(\pi) \in \{1, 0, -1\}$ classifying the representations into real, complex and quaternion. The classical result of character theory of finite groups  proves that the number of involutions of a finite group coincides with the sum of degrees of its irreducible real representations (Corollary 4.6, \cite{MR2270898}*{pp. 81}).


    During the last twenty years, substantial progress has made towards formulating more general versions of these invariants. A significant development in this direction was contributed by Bump and Ginzburg \cite{MR2068079}, who extended the foundational work of Mackey \cite{MR100035}, and Kawanaka–Matsuyama \cite{MR1078503}, introduced  a variant of Frobenius–Schur indicators twisted by underlying group automorphisms. These twisted indicators play a significant role in the analysis of multiplicity-free permutation representations, the construction of models for finite groups (as described in \cite{MR414792}), and in understanding Shintani’s character lifting for finite reductive groups.

    Gow \cite{MR716800} proved the following result, for an odd prime $q$, the number of symmetric matrices in $\mathrm{GL}(n, \mathbb{F}_q)$ is equal to the sum of the degrees of all its irreducible representations of $\mathrm{GL}(n, \mathbb{F}_q)$ and similar result for $GSp(2n,\mathbb{F}_q)$ due to Vinroot \cite{MR2173976}.

     Extending the Frobenius and Schur work to other automorphisms and complex representations a conjecture was proposed by Des Machale which is noted by Khukhro and Mazurov in \cite{MR1392713}*{Problem 16.60}. As far as literature is concerned, we could not find any specific information on this conjecture.

    We are motivated to work on groups of order $p, 2p, p^2$ which are elementary groups. And the dihedral groups appear while the study of groups of order $2p$ due to Sylow's Theorems \cite{MR2286236}.
    The Automorphism group of dihedral groups $\Aut(D_l)$ is well studied and completely characterized. Furthermore, all the groups of order $p, p^2, 2p$ are classified that each is isomorphic to a cyclic group or direct product of cyclic group in the abelian case and isomorphic to dihedral groups in the non-abelian case. Moreover, Cayley proved that every finite group $G$ can be identified by a subgroup of symmetric group $\mathfrak{S}_l$ which is noted by Artin in \cite{MR1129886} which plays a major role in extending concepts to other non-abelian finite groups.

    As in the case of real representations, we have identity automorphism, similar to that we expect that the sum of degrees of all its irreducible complex representations of subgroup of symmetric group will coincide with the number of involutions for some non-identity automorphism. This will be explored in a future article. Therefore, this paper would be the starting point for understanding the subgroup $D_l$ of $\mathfrak{S}_l$ and plays a very important role as underlined above and our main aim is to study this structure in this paper.

    We begin by reviewing some basic definitions and facts from group theory and number theory that will be used throughout in this work. Unless otherwise specified, the terminology, notations, and preliminary results follows the conventions as mentioned in \cite{MR1129886} and \cite{MR2286236}. For a finite group $G$, we write $|G|$ for its order, and for an element $g \in G$, the notation $\mathrm{ord}(g)$ refers to the order of $g$ and $T(G)$ represent the sum of degrees of all its irreducible complex representations. We denote their greatest common divisor and least common multiple by $\gcd(x_1, x_2)$ and $\lcm(x_1, x_2)$, respectively, for integers $x_1, x_2$.
    
    \begin{proposition}[Corollary 2.8.11, \cite{MR1129886}*{pp. 58}]
        Let $p$ be a prime number and suppose that $G$ is a group with $|G| = p$. Then $G$ must be generated by a single element, in particular, $G$ is cyclic.
    \end{proposition}
    \begin{corollary}\label{prime_order_is_abelian}
        Let $G$ be a group whose order is prime number. Then $G$ must be an abelain group.
    \end{corollary}
    \begin{proposition}[Proposition 7.3.3, \cite{MR1129886}*{pp. 198}]\label{p^2_is_abelian}
        Let $p$ be a prime number and $G$ be a group such that $|G| = p^2$. Then $G$ is abelian.
    \end{proposition}
    The next proposition is an immediate consequence of the Sylow theorems. For a detailed study, see $\S 4.5$ of \cite{MR2286236}.
    \begin{proposition}[\S 4.5, \cite{MR2286236}*{pp. 139}]\label{group_of_order_2p}
        Let $p$ be a prime number and $G$ be a group of order $2p$. Then $G$ is isomorphic either to the cyclic group $\mathbb{Z}_{2p}$ or to the dihedral group $D_p$.
    \end{proposition}

    Some well known lemmas are listed below and play a pivotal role in our study.
    \begin{lemma}\label{automorphism_characterization_of_D_l}
        If $l \geq 3$, then $\mathrm{Aut}(D_{l}) = \{f_{a,b} : 0 \leq a \leq l - 1, 1 \leq b \leq l - 1, \gcd(b, l) = 1\}$ where $f_{a,b}: D_{l} \to D_{l}$ is the group homomorphism defined by $f_{a, b}(r) = r^b$, $f_{a,b}(s) = r^as$.
    \end{lemma}
        

    \begin{lemma}\label{gcd_lcm_inequality}
        Let $x_1, x_2$ be positive integers. Then $x_1 + x_2 \le \mathrm{lcm}(x_1,x_2) + \gcd(x_1,x_2)$.
    \end{lemma}
    \begin{proof}
        Write $x_1 = da_1$ and $x_2 = da_2$, where $d = \gcd(x_1, x_2)$ and $\gcd(a_1, a_2) = 1$. Then $\mathrm{lcm}(x_1, x_2) = da_1a_2$. The inequality becomes 
        \begin{align*}
            da_1 + da_2 \leq da_1a_2 + d \iff a_1 + a_2 \leq a_1a_2 + 1
        \end{align*}
        Since $a_1, a_2$ are positive integers, we have $(a_1 - 1)(a_2 -1) \geq 0$, which is equivalent to $a_1 + a_2 \leq a_1a_2 + 1$. This proves the inequality.
    \end{proof}
    \begin{proposition}\label{gcd_main_supporting_result}
        Let $x_1, x_2, x_3$ be positive integers. Then $\gcd (x_1, x_2) + \gcd (x_3, x_2) \leq x_2 + \gcd (x_1,x_3)$.
    \end{proposition}
    \begin{proof}
        Let $g_1 = \gcd(x_1,x_2),\ g_2 = \gcd(x_3,x_2),\ g_3 = \gcd(x_1,x_3)$, then $g_1 \mid x_2$ and $g_2 \mid x_2$ and there exists positive integers $u_1, u_2$ such that $x_2 = g_1u_1$ and $x_2 = g_2u_2$ which implies that $x_2$ is common multiple of $g_1, g_2$, so that
        \begin{equation}\label{eq1}
            \mathrm{lcm}(g_1,g_2) \leq x_2.
        \end{equation}
        Moreover, any common divisor of $g_1$ and $g_2$ divides both $x_1$ and $x_3$. Thus,
        \begin{equation}\label{eq2}
            \gcd(g_1,g_2) \le g_3.
        \end{equation}
        Now, on applying the Lemma \ref{gcd_lcm_inequality}, to $g_1, g_2$, we obtain
        \begin{align*}
            g_1 + g_2 \le \mathrm{lcm}(g_1,g_2) + \gcd(g_1,g_2).
        \end{align*}
        Finally, on using \eqref{eq1} and \eqref{eq2}, we get $g_1 + g_2 \le b + g_3.$, i.e., $\gcd (x_1, x_2) + \gcd (x_3, x_2) \leq x_2 + \gcd (x_1,x_3)$. This completes the proof.
    \end{proof}
    \begin{proposition}\label{gcd_secondary_supporting_result}
        Let $x_1$ be a natural number. Then
        \begin{align*}
            \gcd(x_1-1,x_1+1)=
            \begin{cases}
                2, & \text{if $x_1$ is odd},\\[4pt]
                1, & \text{if $x_1$ is even}.
            \end{cases}
        \end{align*}
    \end{proposition}
    \begin{proof}
        Since $\gcd(x_1-1,x_1+1)=\gcd\bigl(x_1-1,(x_1+1)-(x_1-1)\bigr)=\gcd(x_1-1,2)$, so $\gcd(x_1-1,x_1+1)$ is either $1$ or $2$. If $x_1$ is even, then $x_1 - 1$ is odd, therefore $\gcd(x_1-1,2) = 1$. If $x_1$ is odd, then $x_1-1$ is even, so $\gcd(x_1 - 1, 2) = 2$. This completes our argument.
    \end{proof}
    \begin{definition}
        Let $\alpha$ be an automorphism of $G$, an element $x$ of $G$ is called an \emph{involution} of $\alpha$ if $\alpha(x) = x^{-1}$.
    \end{definition}
    Henceforth, unless stated otherwise, we denote $m_{\alpha}$ by the number of involutions of the automorphism $\alpha$ corresponding to the group $G$. 
    \begin{proposition}[\S 5.3, \cite{MR450380}*{pp. 36}]\label{number_of_irrep_equals_conjugacy_classes}
        Let $G$ be a finite group. Then the number of conjugacy classes of a group coincides with the number of equivalence classes of its irreducible representations.
    \end{proposition}
    \begin{proposition}[Corollary 4.1.8, \cite{MR2867444}*{pp. 29}]\label{sum_of_irreducibles_of_abelian_group}
        All irreducible representations of a finite abelian group have degree one.
    \end{proposition}
    \begin{proposition}[\S 5.3, \cite{MR450380}*{pp. 36}]\label{degrres_irr_of_dn}
        Let $D_l$ be the dihedral group of order $2l$. Then 
        \begin{enumerate}
            \item If $l = 2k$, then $D_l$ has four 1-degree complex irreducible representations and $(k - 1)$ 2-degree complex irreducible representations.
            \item If $l = 2k + 1$, then $D_l$ has two 1-degree complex irreducible representations and $k$ 2-degree complex irreducible representations.
        \end{enumerate}
    \end{proposition}
    An important consequence of Proposition \ref{degrres_irr_of_dn}, is that it provides the exact sum of the degrees of all irreducible complex representations of dihedral groups.
    \begin{corollary}\label{sum_of_irrep_dimensions}
        The sum of degrees of complex irreducible representations of $D_l$ is $l + 2$ when $l$ is even, and $l + 1$ when $l$ is odd.
    \end{corollary}

    \section{Main Results}
        While studying the groups of order $p, 2p, p^2$. The only situation in which a non-abelian group can appear is when the order is $2p$, in that case, the group is isomorphic to the dihedral group $D_p$. We start our analysis on the involutions of identity automorphism and its relation with sum of degrees of all irreducible complex representations and proceed to other automorphisms. Since the examination for finite abelian groups becomes easy to due to Proposition \ref{sum_of_irreducibles_of_abelian_group}.
        
         \begin{lemma}\label{identity_involutions}
            The identity automorphism of dihedral group $D_l$ has exactly $l + 2$ involutions when $l$ is even, and $l + 1$ involutions when $l$ is odd.
        \end{lemma}
        \begin{proof}
            Since every element of $D_l$ can be expressed uniquely in one of the forms $r^k$ or $r^ks$, where $0 \le k < l$. We first determine when such elements are involutions. Suppose the rotation $r^k$ for some $k$, is an involution, then $k$ satisfies $2k \equiv 0 \pmod{l}$. There are two possible cases: $l$ is even or $l$ is odd.
            
            \begin{enumerate}
                \item[Case (i):]  $l$ is even, then the congruence has two solutions, precisely $k = 0, \frac{l}{2}$.
                \item[Case (ii):] $ l$ is odd, then the congruence has a unique solution, i.e., $k = 0$.
            \end{enumerate}
            Since, every reflection is an element of order $2$, so that, there are precisely $l$ reflections in $D_l$. Therefore the total number of involutions of identity automorphism of $D_l$ are
            \begin{align*}
                \begin{cases}
                    l + 2, & \text{if $l$ is even},\\[4pt]
                    l + 1, & \text{if $l$ is odd}.
                \end{cases}
            \end{align*}
        \end{proof}
        Combining Corollary \ref{sum_of_irrep_dimensions}, with Lemma \ref{identity_involutions}, yields the following lemma.
        \begin{lemma}\label{sum_of_irrep_equal_to_identity_in_Dn}
            Let $e$ denote the identity automorphism of $D_l$. Then $T(D_l) = m_e$
        \end{lemma}
        On using Proposition \ref{sum_of_irreducibles_of_abelian_group}, and Lemma \ref{sum_of_irrep_equal_to_identity_in_Dn}, our attention is reduced to the non-identity automorphisms of dihedral groups. Now, we state and prove the some necessary results that pave a path to prove our main result.
       \begin{lemma}\label{identity_greater_other}
            Let $\sigma$ be an automorphism of $D_l$ and $e$ denote the identity automorphism of $D_l$. Then $m_e \geq m_{\sigma}$.
        \end{lemma}
        \begin{proof}
            
             Since reflection is of the form \(r^ks \), where \(0\le k<l\). On employing Lemma \ref{automorphism_characterization_of_D_l}, we have
             \[
                \varphi_{a,b}(r^ks)=\varphi(r^k)\varphi(s) = r^{ak}\cdot r^bs  = r^{ak+b}s.
             \]
            Thus, \(r^ks\) is an involution iff \(r^{ak+b}s =  r^ks\), i.e.,
            \[
                r^{ak+b} = r^k \quad\Longleftrightarrow\quad (a-1)k \equiv -b \pmod l. \tag{1}
            \]
            Suppose $d_1=\gcd(a-1,l)$, the congruence \((a-1)k\equiv -b\pmod l\) has a solution iff \(d_1\mid b\), and has exactly \(d_1\) incongruent solutions modulo \(l\). So that, the number of reflections that are involutions equals \(d_1\) when \(d_1\mid b\), and equals \(0\) otherwise.
            And every rotation is of the form \(r^k\), where \(0\le k<l\). Then $\varphi_{a,b}(r^k)=\varphi(r)^k=r^{ak}$ and \(r^k\) is an involution iff \(r^{ak} = r^{-k}\), i.e.,
            \[
                r^{ak} = r^{-k} \quad\Longleftrightarrow\quad (a+1)k \equiv 0 \pmod l. \tag{2}
            \]
            Now, let \(d_2=\gcd(a+1,l)\), then the congruence \((a+1)k\equiv 0\pmod l\) has a solution, because \(d_2\mid 0\), and has exactly \(d_2\) incongruent solutions modulo \(l\). So that, the number of rotations that are involutions are $d_2$. Now, on summing up the maximum number of reflections that are involutions of $\varphi_{a, b}$ and maximum number of rotations that are involutions of $\varphi_{a, b}$, we obtain $\gcd(a - 1, l) + \gcd(a + 1, l)$. Therefore, the result follows from Lemma \ref{gcd_main_supporting_result}, and Lemma \ref{gcd_secondary_supporting_result}. This completes our argument.
        \end{proof}
    
        An immediate consequence of Lemma \ref{identity_greater_other}, together with Lemma \ref{sum_of_irrep_equal_to_identity_in_Dn}, is the following corollary.
        \begin{corollary}\label{main_theorem}
            Let $\alpha\in \Aut(D_l) $, consider the set $S_{\alpha} = \{g \in D_l \ | \ \alpha(g) = g^{-1}\}$ corresponding to $\alpha$, then $T(D_l) \geq |S_{\alpha}|$.
        \end{corollary}
        \begin{example}
            The automorphisms of $D_3$ and its corresponding involutions are listed as follows.
            \begin{table}[H]
                \centering
                \caption{The Automorphisms of $D_3$}
                \label{tab:aut_d3}
                \begin{tabular}{@{} |c c c l| @{}} 
                    \toprule
                    \textbf{Automorphism ($\sigma$)} & $\boldsymbol{\phi(r)}$ & $\boldsymbol{\phi(s)}$ & \textbf{$S_{\sigma}, m_{\sigma}$} \\
                    \midrule
                    e & $r$ & $s$ & $\{s, rs, r^2s, e\}$, 4 \\
                    $\sigma_1$ & $r$ & $rs$ & $\{e\}$, 1 \\
                    $\sigma_2$ & $r$ & $r^2s$ & $\{e\}$, 1 \\
                    \midrule
                    $\sigma_3$ & $r^2$ & $s$ & $\{e, r, r^2, s\}$, 4 \\
                    $\sigma_4$ & $r^2$ & $sr$ & $\{e, r, r^2, r^2s\}$, 4 \\
                    $\sigma_5$ & $r^2$ & $sr^2$ & $\{e, r, r^2, rs\}$, 4 \\
                    \bottomrule
                \end{tabular}
            \end{table}
            And, we have one irreducible complex representation of degree two and two irreducible complex representations of degree one. We can see that $T(D_3) = 4$ is always greater than or equal to $m_{\sigma}$, for all $\sigma \in \Aut(D_3)$.
        \end{example}
        \begin{lemma}\label{abelian_final_result}
            Let $\alpha$ denote an automorphism of an abelian group $G$. Then $T(G) \geq m_{\alpha}$.
        \end{lemma}
        \begin{proof}
            Since involutions are the elements of the group $G$, so $m_{\sigma} \leq |G|$. From Propositions \ref{number_of_irrep_equals_conjugacy_classes}, and \ref{sum_of_irreducibles_of_abelian_group}, it follows that $|G| = T(G)$. This completes our proof.
        \end{proof}
        \begin{theorem}\label{final_theorem}
            Let $G$ be a finite group whose order is $2p, p$, or $p^2$. Consider the set $S_{\alpha} = \{g \in G \ | \ \alpha(g) = g^{-1}\}$ corresponding to $\alpha \in \Aut(G)$, then $T(G) \geq |S_{\alpha}|$ for all $\alpha \in \Aut(G)$.
        \end{theorem}
        \begin{proof}
            Given $\alpha \in \Aut(G)$ and the set $S_{\alpha}$ corresponding to $\alpha$. Let $S_e$ denote
            the set corresponding to identity automorphism $e : G \rightarrow G$ defined as $e(x) = x $ for all $x \in G$.
            \begin{align*}
                S_e = \{y \in G \ | \ \text{ord(y)} = 2  \text{ or } \text{ord(y) } = 1\} 
            \end{align*} Then for a prime number $p$, we have the following cases:
            \begin{enumerate}
                \item[Case (i).] When $|G|$ is either $p$ or $p^2$, Corollary \ref{prime_order_is_abelian}, together with Proposition \ref{p^2_is_abelian}, ensures that group is abelian. The desired conclusion then follows from Lemma \ref{abelian_final_result}.
                \item[Case (ii).] When $ |G| = 2p$, Proposition \ref{group_of_order_2p}, implies that $G$ is either abelian or isomorphic to the dihedral group $D_p$.
                \begin{enumerate}
                	\item If $G$ is abelian, then it follows from Lemma \ref{abelian_final_result}, as required.
                	\item If $G$ is non-abelian, then $G$ is isomorphic to $D_p$, and the claim follows from Corollary \ref{main_theorem}.
            	\end{enumerate}
            \end{enumerate}
            This completes our argument.
        \end{proof}
        \begin{remark}
            In fact, Theorem \ref{final_theorem}, remains valid for any finite abelian group. Moreover, the equality $|S_e| = T(G)$ does not hold in general.
        \end{remark}    \begin{example}
            Consider the group $\mathbb{Z}_3$. The only involution of $\mathbb{Z}_3$ under the identity automorphism is $\overline{0}$, whereas $T(\mathbb{Z}_3) = 3$.
        \end{example}

        \bibliography{ref}
        \bibliographystyle{amsplain}
\end{document}